\documentclass[a4paper,12pt]{scrartcl}

\usepackage[utf8]{inputenc}
\usepackage{amsfonts}
\usepackage{amsmath}
\usepackage{amssymb}
\usepackage{graphicx}

\usepackage{color}

\usepackage{a4wide}
\usepackage{empheq}
\numberwithin{equation}{section}

\usepackage{amsthm}
\theoremstyle{plain}
\newtheorem{thm}{Theorem}[section]
\newtheorem{lem}[thm]{Lemma}
\newtheorem{cor}[thm]{Corollary}
\newtheorem{prop}[thm]{Proposition}
\newtheorem{obs}[thm]{Observation}
\newtheorem{fact}[thm]{Fact}

\newtheorem*{MainResult}{Main Result}
\newtheorem*{ConvexSets}{Theorem~\ref{thm:strong_convex_sets_euclidean}}

\theoremstyle{definition}

\newtheorem{exmpl}[thm]{Example}

\theoremstyle{remark}
\newtheorem{rem}[thm]{Remark}

\newcommand{\R}{\mathbb{R}}

\usepackage{ifthen}
\newcommand{\optionalsubsuper}[3]{%
\ifthenelse{\equal{#2}{}}{%
  \ifthenelse{\equal{#3}{}}{%
    #1}{%
    #1^{#3}}
  }{%
  \ifthenelse{\equal{#3}{}}{%
    #1_{#2}}{%
    #1_{#2}^{#3}}}}

\newcommand{\nin}{\notin}
\newcommand{\union}{\cup}
\newcommand{\Union}{\bigcup}
\newcommand{\intersect}{\cap}

\newcommand{\into}{\hookrightarrow}
\newcommand{\onto}{\twoheadrightarrow}

\newcommand{\defeq}{\mathrel{\mathop{:}}=}
\newcommand{\eqdef}{=\mathrel{\mathop{:}}}

\newcommand{\abs}[1]{\left\lvert#1\right\rvert}

\newcommand{\Link}[1]{\optionalsubsuper{\operatorname{lk}}{#1}{}}

\newcommand{\CAT}[1]{\ensuremath{\textrm{CAT}(#1)}}

\newcommand{\overbar}[1]{\mkern 1.5mu\overline{\mkern-1.5mu#1\mkern-1.5mu}\mkern 1.5mu}

\newcommand{\Space}{X}
\newcommand{\AltSpace}{Y}
\newcommand{\Point}{a}
\newcommand{\AltPoint}{b}
\newcommand{\YetAltPoint}{c}
\newcommand{\FourthPoint}{d}

\newcommand{\Ball}[1]{B_{#1}}
\newcommand{\CBall}[1]{\overbar{B}_{#1}}
\newcommand{\ClosedSet}{A}
\newcommand{\Apartment}{\Sigma}

\newcommand{\Chamber}{C}
\newcommand{\ChamberAtInfinity}{c}
\newcommand{\Retraction}[2]{\rho_{#1,#2}}
\newcommand{\Path}{\gamma}
\newcommand{\Cell}{\sigma}
\newcommand{\id}{\operatorname{id}}
\newcommand{\Proj}[1]{\operatorname{pr}_{#1}}
\newcommand{\Infty}[1]{{#1}^\infty}

\title{Local convexity in \CAT{\kappa}-spaces}
\author{Kai-Uwe Bux and Stefan Witzel}

\begin{document}

\maketitle

\begin{abstract}
\noindent
Heinrich Tietze has shown that for a closed connected subset of euclidean space being convex is a local property. We generalize this to \CAT{0}-spaces and locally compact \CAT{\kappa} spaces. As an application we give a construction of certain convex sets in euclidean buildings.
\end{abstract}

Let $\Space$ be a \CAT{0}-space, that is, a simply connected geodesic metric space with nonpositive curvature in the sense that triangles are at most as thick as their comparison triangles in euclidean space. We say that a closed subset $\ClosedSet \subseteq \Space$ is \emph{locally convex at $\Point \in \ClosedSet$} if there is an $\varepsilon > 0$ such that $\ClosedSet \intersect \Ball{}(x,\varepsilon)$ is convex in $\Ball{}(x,\varepsilon)$ and that $\ClosedSet$ is \emph{locally convex} if it is locally convex at every point. Our main result is

\begin{MainResult}
Let $\Space$ be a complete \CAT{\kappa}-space. Let $\ClosedSet \subseteq \Space$ be a closed and connected subset. If $\kappa > 0$ assume further that $X$ is locally compact and that the diameter of $\ClosedSet$ is at most $D_\kappa$ in the length metric $d_\ClosedSet$. If $\ClosedSet$ is locally convex then it is $D_\kappa$-convex.
\end{MainResult}

It is proven as Theorem~\ref{thm:local_convexity} (if $\Space$ is locally compact $\CAT{\kappa}$) and Theorem~\ref{thm:local_convexity_cat0} (if $\Space$ is \CAT{0}). In the case $\kappa > 0$ the assumption on the diameter of $\ClosedSet$ will be seen to be necessary, however the local compactness assumption is an artifact of the proof. In fact, the Main Result was proven without local compactness assumption by Carlos Ramos-Cuevas \cite{ramos-cuevas13}.

In the case where $\Space$ is a euclidean space, Theorem~\ref{thm:local_convexity_cat0} has been shown by Tietze in \cite{tietze28}, see also \cite{schoenberg48} and \cite[(5.2)]{klee51}. Papadopoulos~\cite[Theorem~8.3.3]{papadopoulos05} has a version of Theorem~\ref{thm:local_convexity_cat0} for locally compact Busemann spaces whose proof is very similar to that of Theorem~\ref{thm:local_convexity}. In \cite[§24~(b)]{gromov01} Gromov formulates a statement that implies Theorem~\ref{thm:local_convexity_cat0}. The method for proving Theorem~\ref{thm:local_convexity_cat0} was also used by Sahattchieve in \cite[Proposition~2.14]{sahattchieve12}.

\medskip

As an application we show a certain class of subsets of euclidean buildings to be convex. An alternative approach to this application, which is independent of the convexity criteria, has been suggested to the authors by Koen Struyve.

The description of the convex subsets of buildings uses the following technical condition. Let $\Apartment$ be a euclidean Coxeter complex and let $\Chamber \subseteq \Apartment$ be a chamber or $\Chamber \subseteq \Infty \Apartment$ be a chamber at infinity. A closed, convex set $\ClosedSet \subseteq \Apartment$ is said to \emph{satisfy the weak normal condition} if for every boundary point $\Point \in \partial \ClosedSet$ and every wall $H$ containing $\Point$, there is a vector at $\Point$ normal to $\ClosedSet$ that points into the halfspace of $H$ that does not contain $\Chamber$.
Recall also that if a building $\Space$ contains $\Apartment$ as an apartment, then there is a retraction of $\Retraction{\Apartment}{\Chamber} \colon \Space \to \Apartment$ of $\Space$ onto $\Apartment$ centered at $\Chamber$.

\begin{ConvexSets}
Let $\Space$ be a euclidean building, let $\Apartment \subseteq \Space$ be an
apartment and let $\Chamber \subseteq \Apartment$ be a chamber or $\Chamber
\subseteq \Infty\Apartment$ be a chamber at infinity. Let $\ClosedSet \subseteq
\Apartment$ be closed, convex and assume that $\ClosedSet \intersect
\Chamber \ne \emptyset$ (respectively $\Infty\ClosedSet \intersect \Chamber \ne
\emptyset$). Assume further that $\ClosedSet$ satisfies the weak normal condition.
Then $\tilde{\ClosedSet} \defeq
\Retraction{\Apartment}{\Chamber}^{-1}(\ClosedSet)$ is convex.
\end{ConvexSets}

The paper is organized as follows. The convexity criteria are proved in Section~\ref{sec:local_convexity}.
In Section~\ref{sec:convex_sets_in_euclidean_buildings} we use them to prove a weaker version of Theorem~\ref{thm:strong_convex_sets_euclidean}. Finally, in Section~\ref{sec:alternative_proof} we give the alternative proof of Theorem~\ref{thm:strong_convex_sets_euclidean} that is independent of Sektion~\ref{sec:local_convexity}.

\medskip

We want to thank Peter Abramenko for helpful discussions about questions that
lead to this article and Koen Struyve for the remark that led to
Section~\ref{sec:alternative_proof}. The second author thanks the University of
Virginia for the hospitality he enjoyed during the research for this article
and greatfully acknowledges financial support by the DFG.

\section{Convex sets in \CAT{\kappa}-spaces}
\label{sec:local_convexity}

Let $(\Space,d)$ be a metric space. We will use the following definitions which are taken from \cite{brihae}.

A map $\Path \colon [a,b] \to \Space$ is a \emph{geodesic} if it is an isometric embedding. It is a \emph{local geodesic at~$\Path(t)$} if there is an $\varepsilon > 0$ such that $\Path$ preserves distances on $\Ball{}(t,\varepsilon)$ and it is a \emph{local geodesic} if it is a local geodesic at every point.

A map $f \colon \Space \to \AltSpace$ is \emph{locally an isometric embedding} if for every $\Point \in \Space$ there is a $\varepsilon > 0$ such that the restriction of $f$ to $\Ball{}(\Point,\varepsilon)$ is an isometric embedding, i.e.\ an isometry onto its image.

The space $\Space$ is \emph{(uniquely) geodesic} if for any two points there is a (unique) geodesic joining them. It is a \emph{length space} if the distance between two points is the infimum over the lengths of paths that join them. It is \emph{proper} if its closed bounded sets are compact. It is a \CAT{\kappa}-space if triangles are at most as thick as their comparison triangles in a space of curvature $\kappa$, see \cite[Definition~II.1.1]{brihae}. Recall that $D_\kappa = \pi/\sqrt{\kappa}$ if $\kappa > 0$ and $D_\kappa = \infty$ otherwise.

\medskip

We sketch the proof of the convexity criterion for locally compact \CAT{\kappa} spaces, see Figure~\ref{fig:argument_sketch}. The crucial ingredient is the observation that being a geodesic is a local property, see \cite[Proposition~II.1.4~(2)]{brihae}:

\begin{prop}
\label{prop:local_geodesic}
Let $\Space$ be a \CAT{\kappa}-space. A path in $\Space$ of length at most $D_\kappa$ is a geodesic if and only if it is a local geodesic.
\end{prop}

\begin{figure}[htb]
\begin{center}
\includegraphics{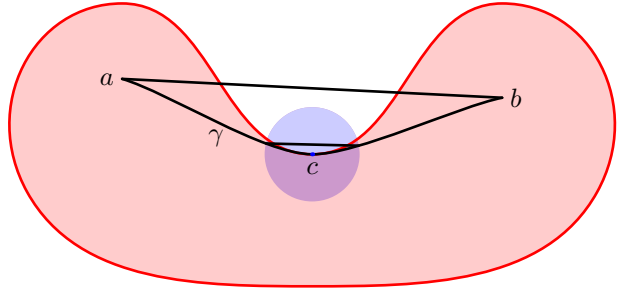}
\end{center}
\caption{The idea of proof of Theorem~\ref{thm:local_convexity}.}
\label{fig:argument_sketch}
\end{figure}

Let $\Space$ be a \CAT{0}-space. Let $\ClosedSet \subseteq \Space$ be a closed subset and assume that $\ClosedSet$ is not convex so that there are points $\Point$ and $\AltPoint$ such that the geodesic $[\Point,\AltPoint]$ is not fully contained in $\ClosedSet$. Then the path $\Path$ that is shortest among the paths from $\Point$ to $\AltPoint$ that are fully contained in $\ClosedSet$ cannot be a geodesic. Hence there has to be a point $\YetAltPoint$ of $\Path$ at which it is not a local geodesic by Proposition~\ref{prop:local_geodesic}. But then $\ClosedSet$ cannot be locally convex at $\YetAltPoint$. Of course we have to ascertain that a shortest path $\Path$ as above exists. This will follow from the Hopf--Rinow Theorem, \cite[Proposition~I.3.7~(2)]{brihae}:

\begin{prop}
\label{prop:hopf-rinow}
A complete, locally compact length space is a proper geodesic space.
\end{prop}

If $\Space$ is a geodesic metric space and $\ClosedSet \subseteq \Space$ is a closed subset, we let $d_\ClosedSet$ denote the length metric on $\ClosedSet$ induced by $d$. That is
\[
d_A(\Point,\AltPoint) = \inf_{\Path} l(\Path)
\]
where the infimum is taken over rectifiable paths $\gamma$ in $\ClosedSet$ that join $\Point$ to $\AltPoint$.

In general, passage to the length metric does not preserve desirable properties as the following example shows.

\begin{exmpl}
Let $\Space \defeq \R^3$ and let $K \subseteq \Space$ be a Koch snowflake (see
\cite[Section~2.4]{peijursau04}). Let $\Point$ be a point outside the plane
spanned by $K$ let $\ClosedSet$ be the cone over $K$ with cone point~$\Point$
(that is, $\ClosedSet = \Union_{k \in K} [\Point,k]$). Clearly $\ClosedSet$ is compact with
respect to $d$. But if $\AltPoint, \AltPoint'$ are two distinct points in $K$,
the shortest rectifiable path in $\ClosedSet$ connecting them is the path from
$\AltPoint$ to $\Point$ composed with the path from $\Point$ to $\AltPoint'$.
Hence every segment of the form $(\Point,k]$ with $k \in K$ is open in
$(\ClosedSet,d_\ClosedSet)$. Thus $\ClosedSet$ is not locally compact.
\end{exmpl}

However it turns out that the assumption that $\ClosedSet$ be locally convex suffices to avoid these cases:

\begin{obs}
\label{obs:inclusion_local_isometry}
Let $(\Space,d)$ be a metric space and let $\ClosedSet \subseteq \Space$ be locally convex. The inclusion $(\ClosedSet,d_\ClosedSet) \into (\Space,d)$ is a locally an isometric embedding.
\end{obs}

\begin{proof}
Let $\Point \in \ClosedSet$ be arbitrary and let $\varepsilon > 0$ be such that $\ClosedSet \intersect \Ball{\Space}(\Point,\varepsilon)$ is convex in $\Ball{\Space}(\Point,\varepsilon)$. Then the restrictions to $\ClosedSet \intersect \Ball{\Space}(\Point,\varepsilon)$ of $d$ and $d_\ClosedSet$ coincide.
\end{proof}

\begin{cor}
\label{cor:closed_locally_compact_nice_in_length_metric}
Let $(\Space,d)$ be a complete, locally compact metric space and let $\ClosedSet \subseteq \Space$ be closed and locally convex. Then $\ClosedSet$ is complete and locally compact with respect to the length metric $d_\ClosedSet$. If $\ClosedSet$ is connected, then it is connected in the length metric.
\end{cor}

\begin{proof}
Clearly $(\ClosedSet,d)$ is complete and locally compact.
Since completeness and local compactness are local properties, it follows from Observation~\ref{obs:inclusion_local_isometry}, that $(\ClosedSet,d_\ClosedSet)$ is complete and locally compact.

The set of points that can be reached by rectifiable paths from a given point is open and closed. So if $(\ClosedSet,d)$ were not connected by rectifiable paths, it would not be connected.
\end{proof}

The argument for connectedness was already used in \cite[Lemma~3]{schoenberg48} and \cite[(5.1)]{klee51}. 

We can now prove the first convexity criterion.

\begin{thm}
\label{thm:local_convexity}
Let $\Space$ be a complete, locally compact \CAT{\kappa}-space. Let $\ClosedSet \subseteq \Space$ be a closed and connected subset. If $\kappa > 0$ assume further that the diameter of $\ClosedSet$ is at most $D_\kappa$ in the length metric $d_\ClosedSet$. If $\ClosedSet$ is locally convex then it is $D_\kappa$-convex.
\end{thm}

\begin{proof}
By Corollary~\ref{cor:closed_locally_compact_nice_in_length_metric}, $(\ClosedSet,d_\ClosedSet)$ is a complete locally compact length space, so it is a proper geodesic space by Proposition~\ref{prop:hopf-rinow}. Since $\ClosedSet$ is assumed to be connected, any two points have finite distance.

Let $\Point,\AltPoint \in \ClosedSet$ have distance $d(\Point,\AltPoint) < D_\kappa$ and let $\Path \colon [0,l] \to \ClosedSet$ be a geodesic joining them in $(\ClosedSet,d_\ClosedSet)$. By assumption the length of $\Path$ is at most $D_\kappa$.

Suppose that $\Path$ is not a geodesic in $(\ClosedSet,d)$. Then by
Proposition~\ref{prop:local_geodesic} there is a $t \in [0,l]$ such that $\Path$
is not a local geodesic at $\Path(t)$. But by assumption $\ClosedSet$ is locally
convex at $\Path(t)$ so there is an open neighborhood $U$ such that $\ClosedSet
\intersect U$ is convex in $U$. That is, $d$ and $d_\ClosedSet$ coincide on
$\ClosedSet \intersect U$ and for any two points there is a geodesic that joins
them. Since $\Path$ is not a local geodesic at $\Path(t)$, there are $t',t'' \in
[0,l]$ such that $\gamma([t',t'']) \subseteq U$ and
$d_\ClosedSet(\gamma(t'),\gamma(t'')) < \abs{t' - t''}$. Replacing
$\Path|_{[t',t'']}$ by the geodesic from $\gamma(t')$ to $\gamma(t'')$ yields a
path from $\Point$ to $\AltPoint$ in $(\ClosedSet,d_\ClosedSet)$ that is
strictly shorter than $\Path$ contradicting the assumption that $\Path$ is a
geodesic in $(\ClosedSet,d_\ClosedSet)$.
\end{proof}

Note that the diameter of $\ClosedSet$ has to be bounded with respect to the length metric as the following example shows:

\begin{exmpl}
Let $\Space$ be the unit circle, which is a \CAT{1}-space. It has diameter $\pi = D_1$ and hence every subset has diameter at most $D_1$. Let $B$ be an open ball of radius $\pi/4$ in $\Space$ and let $\ClosedSet = \Space \setminus B$. Then $\ClosedSet$ is closed, connected, and locally convex but it is not $\pi$-convex. In fact, the unique geodesic joining the boundary points of $\ClosedSet$ lies in the closure of $B$.
\end{exmpl}


For \CAT{0}-spaces, we can use the Cartan--Hadamard theorem instead of the Hopf--Rinow theorem and dispense with local compactness. Recall that a space is nonpositively curved, if it is locally \CAT{0}. We use the following version of the Cartan--Hadamard theorem (see \cite[Theorem~II.4.1]{brihae}):

\begin{thm}
\label{thm:cartan_hadamard}
Let $\Space$ be a complete, connected, nonpositively curved metric space. Then $\widetilde{\Space}$ is a \CAT{0} space.
\end{thm}

The point here is that $\Space$ is not assumed to be geodesic while $\widetilde{\Space}$ is asserted to be geodesic (cf.\ Remark~II.4.2(2) in \cite{brihae}). We also use \cite[Proposition~II.4.14]{brihae}:

\begin{prop}
\label{prop:universal_cover_isometric_embedding}
Let $\Space$ and $\AltSpace$ be complete connected metric spaces. Assume that $\Space$ is nonpositively curved and that $\AltSpace$ is locally a length space. If $f \colon \AltSpace \to \Space$ is locally an isometric embedding, then $\AltSpace$ is nonpositively curved and every continuous lifting $\tilde{f} \colon \widetilde{\AltSpace} \to \widetilde{\Space}$ of $f$ is an isometric embedding.
\end{prop}

We can now prove the convexity criterion for complete \CAT{0}-spaces that are not necessarily locally compact.

\begin{thm}
\label{thm:local_convexity_cat0}
Let $\Space$ be a complete \CAT{0}-space. A closed connected subset $\ClosedSet \subseteq \Space$ that is locally convex is convex.
\end{thm}

\begin{proof}
By Observation~\ref{obs:inclusion_local_isometry}, the inclusion $(\ClosedSet,d_\ClosedSet) \into (\Space,d)$ is locally an isometric embedding. So Proposition~\ref{prop:universal_cover_isometric_embedding} implies that $(\ClosedSet,d_\ClosedSet)$ is nonpositively curved and the composition
\[
\widetilde{\ClosedSet} \onto \ClosedSet \into \Space = \widetilde{\Space} 
\]
is an isometric embedding. Hence $\ClosedSet$ is simply connected and the inclusion $(\ClosedSet,d_\ClosedSet) \into (\Space,d)$ is an isometric embedding. So $\ClosedSet$ is \CAT{0} by Theorem~\ref{thm:cartan_hadamard} and in particular geodesic (with respect to $d$).
\end{proof}

\section{Application: constructing convex sets in euclidean buildings}
\label{sec:convex_sets_in_euclidean_buildings}

We call a metric space $\Space$ \emph{locally uniquely geodesic} if for every $\Point \in \Space$ there is an $\varepsilon > 0$ such that $\Ball{}(\Point,\varepsilon)$ is uniquely geodesic, that is, for points $\AltPoint,\YetAltPoint \in \Ball{}(\Point,\varepsilon)$ there is a unique geodesic that joins $\AltPoint$ to $\YetAltPoint$. If $\Space$ is geodesic and of bounded curvature then it is locally uniquely geodesic.

If $\ClosedSet$ is a closed subset of a locally uniquely geodesic space $\Space$, we say that $\Point$ is a \emph{cone point} of $\ClosedSet$ if there is an $\varepsilon > 0$ such that $\Ball{}(\Point,\varepsilon)$ is uniquely geodesic and
\begin{equation}
\text{for every }b \in \Ball{}(\Point,\varepsilon)\text{ either }[a,b] \subseteq \ClosedSet\text{ or }[a,b] \intersect \ClosedSet = \{a\} \text{ .}\label{eq:cone_condition}
\end{equation}
We say that $\ClosedSet$ is \emph{polyhedral} if each of its points is a cone point.

A building is a cell complex that can be covered by Coxeter complexes, called apartments, of a certain type subject to some conditions, see \cite{abrbro}. One of the conditions requires that any two points be contained in a common apartment. The building is \emph{spherical} if its apartments are spherical Coxeter complexes, that is, isometric to round spheres. The building is \emph{euclidean} if its apartments are affine Coxeter complexes, that is, isometric to a euclidean space.

In either case the metrics on the apartments fit together to define a metric on the whole building. By a \emph{metric building} we mean either a spherical or a euclidean building with that metric. A spherical building is a \CAT{1}-space, a euclidean building is a \CAT{0}-space.

We collect some features of buildings that we will need later:

\begin{fact}
\label{fact:buildings}
Let $\Space$ be a metric building.
\begin{enumerate}
\item The link of every cell of $\Space$ that is neither empty nor a chamber is a spherical building.\label{item:links_are_buildings}
\item Given an apartment $\Apartment$ of $\Space$ and a chamber $\Chamber \subseteq \Apartment$ there is a map $\Retraction{\Apartment}{\Chamber} \colon \Space \to \Apartment$ called the \emph{retraction onto $\Apartment$ centered at $\Chamber$} such that $\Retraction{\Apartment}{\Chamber}|_\Apartment = \id_\Apartment$ and $\Retraction{\Apartment}{\Chamber}|_{\Apartment'}$ is an isometry for every apartment $\Apartment'$ that contains $\Chamber$. Moreover, $\Retraction{\Apartment}{\Chamber} \circ \Retraction{\Apartment'}{\Chamber} = \Retraction{\Apartment}{\Chamber}$.\label{item:retractions}
\item Let $\Chamber$ be a chamber and let $\Cell$ be an arbitrary cell. There is a unique chamber $\Proj{\Cell}\Chamber \ge \Cell$, called the \emph{projection of $\Chamber$ onto $\Cell$} such that every apartment that contains $\Chamber$ and $\Cell$ also contains $\Proj{\Cell}\Chamber$. If $\Apartment$ is an apartment that contains $\Chamber$, then $\Proj{\Retraction{\Apartment}{\Chamber}\Cell}\Chamber = \Retraction{\Apartment}{\Chamber}(\Proj{\Cell}\Chamber)$.\label{item:projections}
\end{enumerate}
\end{fact}

If $[\Point,\AltPoint]$ is a geodesic segment, we denote by $[\Point,\AltPoint]_\Point$ the direction that it defines in $\Link{}\Point$. If $\Point$ is a point and $\Chamber$ is a chamber we write $\Proj{\Link{\Space}\Point}\Chamber$ for the directions of $\Link{\Space}\Point$ that point into $\Proj{\Cell}\Chamber$, where $\Cell$ is the carrier of $\Point$ (the smallest cell that contains $\Point$).

\begin{prop}
\label{prop:local_convexity_of_polyhedral_sets}
Let $\Space$ be a metric building and $\ClosedSet \subseteq \Space$ a closed polyhedral subset.
\begin{enumerate}
\item The following are equivalent:\label{item:local_convexity_of_polyhedral_sets}
\begin{enumerate}
\item $\ClosedSet$ is locally convex.
\item For every $\Point \in \ClosedSet$ the subset $\Link{\ClosedSet}\Point$ is $\pi$-convex in $\Link{\Space}\Point$.
\end{enumerate}
\item For every $\Point \in \ClosedSet$ the subset $\Link{\ClosedSet}\Point$ of $\Link{\Space}\Point$ is closed and polyhedral.\label{item:polyhedral_hereditary}
\end{enumerate}
\end{prop}

\begin{proof}
Assume that $\ClosedSet$ is locally convex and let $\Point \in \ClosedSet$ be arbitrary. Let $\varepsilon_1 > 0$ be such that \eqref{eq:cone_condition} is satisfied, let $\varepsilon_2 > 0$ be such that $\ClosedSet \intersect \Ball{}(\Point,\varepsilon_2)$ is convex, let $\varepsilon_3 > 0$ be such that $\Ball{}(\Point,\varepsilon_3)$ is contained in the open star of $\Point$ and set $\varepsilon \defeq \min\{\varepsilon_1,\varepsilon_2,\varepsilon_3\}$. Let $[\Point,\AltPoint]_\Point$ and $[\Point,\YetAltPoint]_\Point$ with $\AltPoint,\YetAltPoint \in \Ball{}(\Point,\varepsilon)$ be directions in $\Link{\ClosedSet}\Point$ that have distance $< \pi$ in $\Link{\Space}\Point$. The condition \eqref{eq:cone_condition} implies that $\AltPoint$ and $\YetAltPoint$ are contained in $\ClosedSet$. The geodesic $[\AltPoint,\YetAltPoint]$ is contained in $\ClosedSet$ because $\ClosedSet \intersect \Ball{}(\Point,\varepsilon)$ is convex. Since $\AltPoint$ and $\YetAltPoint$ are contained in the open star of $\Point$, 
there is an apartment $\Apartment$ that contains $\Point$, $\AltPoint$, and $\YetAltPoint$. In this apartment it is easy to see that the geodesic from $[\Point,\AltPoint]_\Point$ to $[\Point,\YetAltPoint]_\Point$ consists precisely of the directions $[\Point,\FourthPoint]_\Point$ with $\FourthPoint \in [\AltPoint,\YetAltPoint]$. These are contained in $\Link{\ClosedSet}\Point$ by \eqref{eq:cone_condition}.

Conversely we take an arbitrary $\Point \in \ClosedSet$ and assume that $\Link{\ClosedSet}\Point$ is $\pi$-convex in $\Link{\Space}\Point$ and want to show that $\ClosedSet$ is locally convex in $\Point$. Let $\varepsilon_1 > 0$ be such that \eqref{eq:cone_condition} holds, let $\varepsilon_2 > 0$ be such that $\Ball{}(\Point,\varepsilon_2)$ is contained in the open star of $\Point$ and set $\varepsilon \defeq \min\{\varepsilon_1,\varepsilon_2\}$. Let $\AltPoint,\YetAltPoint \in \ClosedSet \intersect \Ball{}(\Point,\varepsilon)$ be arbitrary. If $\angle_\Point(\AltPoint,\YetAltPoint) = \pi/2$, then $[\AltPoint,\YetAltPoint] = [\AltPoint,\Point] \union [\Point,\YetAltPoint]$ (see Lemma~\ref{lem:angle_pi_implies_lying_on_geodesic}). Otherwise let $\Apartment$ be an apartment that contains $\Point$, $\AltPoint$, and $\YetAltPoint$. In this apartment it is easy to see that the geodesic from $[\Point,\AltPoint]_\Point$ to $[\Point,\YetAltPoint]_\Point$ is $\{[\Point,\FourthPoint]_\Point \mid \FourthPoint \in [\
AltPoint,\YetAltPoint]\}$ which by assumption is contained in $\Link{\ClosedSet}\Point$. Thus \eqref{eq:cone_condition} implies that each of the $\FourthPoint \in [\AltPoint,\YetAltPoint]$ lies in $\ClosedSet$.

For the second statement let $\Point \in \ClosedSet$ be arbitrary and let $\varepsilon > 0$ be such that \eqref{eq:cone_condition} holds. Let $\AltPoint \in \Ball{}(\Point,\varepsilon/2)$ be such that $[\Point,\AltPoint]_\Point$ does not lie in $\Link{\ClosedSet}\Point$. Then \eqref{eq:cone_condition} implies that $\AltPoint \nin \ClosedSet$. So there is a $\delta$-ball around $\AltPoint$ that does not meet the closed set $\ClosedSet$ and we take $\delta < \varepsilon/2$. Then \eqref{eq:cone_condition} implies that the open set $\{[\Point,\YetAltPoint]_\Point \mid \YetAltPoint \in \Ball{}(\AltPoint,\delta)\}$ that contains $[\Point,\AltPoint]_\Point$ does not meet $\Link{\ClosedSet}\Point$. This shows that $\Link{\ClosedSet}\Point$ is closed.

It remains to see that $\Link{\ClosedSet}\Point$ is polyhedral. So let $\Point \in \ClosedSet$ be arbitrary and let $\varepsilon_1 > 0$ be such that \eqref{eq:cone_condition} is satisfied. Let $\varepsilon_2 > 0$ be such that $\Ball{}(\Point,\varepsilon_2)$ is contained in the open star of $\Point$ and set $\varepsilon \defeq \min\{\varepsilon_1, \varepsilon_2\}$. Let $[\Point,\AltPoint]_\Point$ with $\AltPoint \in \Ball{}(\Point,\varepsilon/2)$ be an arbitrary direction of $\Link{\ClosedSet}\Point$. We have to show that $[\Point,\AltPoint]_\Point$ is a cone point of $\Link{\ClosedSet}\Point$. Let $\delta' > 0$ be such that \eqref{eq:cone_condition} is satisfied for $\AltPoint$ and $\delta'$ and set $\delta \defeq \min\{\delta,\varepsilon/2\}$. Let $\bar{\delta} > 0$ be such that $\Ball{\Link{\Space}\Point}([\Point,\AltPoint]_\Point,\bar{\delta}) \subseteq \{[\Point,\YetAltPoint]_\Point \mid \YetAltPoint \in \Ball{\Space}(\AltPoint,\delta)\}$. We claim that \eqref{eq:cone_condition} is satisfied for $[\Point,
\AltPoint]_\Point$ and $\bar{\delta}$.

So let $[\Point,\YetAltPoint]_\Point \in \Ball{\Link{\Space}\Point}([\Point,\AltPoint]_\Point,\bar{\delta})$ be arbitrary. By construction we may assume that $\YetAltPoint$ lies in $\Ball{\Space}(\AltPoint,\delta)$. So the condition \eqref{eq:cone_condition} for $\AltPoint$ implies that either $[\AltPoint,\YetAltPoint] \intersect \ClosedSet = \{\AltPoint\}$ or $[\AltPoint,\YetAltPoint] \subseteq \ClosedSet$.

Considering an apartment that contains $\Point$, $\AltPoint$ and $\YetAltPoint$ it is easy to see that the geodesic segment from $[\Point,\AltPoint]_\Point$ to $[\Point,\YetAltPoint]_\Point$ is $\{[\Point,\FourthPoint]_\Point \mid \FourthPoint \in [\AltPoint,\YetAltPoint]\}$. So the condition \eqref{eq:cone_condition} for $\Point$ implies that either all of that segment is contained in $\Link{\ClosedSet}\Point$ or just $[\Point,\AltPoint]_\Point$.
\end{proof}

In the proof we used the following converse to Remark~I.1.13~(2) of \cite{brihae}. See Definition~I.1.12 of loc.cit.\ for the definition of the angle.

\begin{lem}
\label{lem:angle_pi_implies_lying_on_geodesic}
Let $\Space$ be a metric space and let $\Path_1 \colon [0,a_1] \to \Space$ and $\Path_2 \colon [0,a_2] \to \Space$ be two geodesics with $\Path_1(0) = \Path_2(0)$. If $\angle_{\Path_1(0)}(\Path_1,\Path_2) = \pi$ then the path $\Path \colon [-a_1,a_2] \colon \Space$ defined by $\Path(t) = \Path_1(-t)$ for $t \le 0$ and $\Path(t) = \Path_2(t)$ for $t \ge 0$ is a local geodesic.
\end{lem}

\begin{proof}
We only have to show that it is a local geodesic in $0$. If there are $t_1 \in [-a_1,0]$ and $t_2 \in [0,a_2]$ such that $t_2 - t_1 = d(\Path(t_2),\Path(t_1))$, then we are done: in fact, if $\tau_1 \in [t_1,0]$ and $\tau_2 \in [0,t_2]$ are arbitrary (if $\tau_1$ and $\tau_2$ have same sign, then we can just use that $\Path_1$ and $\Path_2$ are geodesics) we have
\begin{eqnarray*}
t_2 - t_1 & = & d(\Path(t_1),\Path(t_2))\\
& \le & d(\Path(t_1),\Path(\tau_1)) + d(\Path(\tau_1),\Path(\tau_2)) + d(\Path(\tau_2),\Path(t_2))\\
& = & (\tau_1 - t_1) + d(\Path(\tau_1),\Path(\tau_2)) + (t_2 - \tau_2)
\end{eqnarray*}
from which we deduce $d(\Path(\tau_1),\Path(\tau_2)) \ge \tau_2 - \tau_1$ and the converse follows from the triangle inequality:
\[
d(\Path(\tau_1),\Path(\tau_2)) \le d(\Path(\tau_1),\Path(0)) + d(\Path(0),\Path(\tau_2)) = \tau_2 - \tau_1 \text{ .}
\]
So if $\Path$ were not a local geodesic, then the map $[-a_1,0] \times [0,a_2] \to \R$ that takes $(t_1,t_2)$ to $t_2-t_1 - d(\Path(t_2),\Path(t_1))$ would have to be nonzero. By compactness it would then have to be bounded away from $0$. But then $\angle_{\Path_1(0)}(\Path_1,\Path_2)$ could not be $\pi$.
\end{proof}

From now on we fix a metric building $\Space$, an apartment $\Apartment \subseteq \Space$, and a chamber $\Chamber \subseteq \Apartment$. Let $\kappa \defeq 1$ if $\Space$ is spherical and $\kappa \defeq 0$ if $\Space$ is euclidean, so that $\Space$ is \CAT{\kappa}. Let $\rho \defeq \Retraction{\Apartment}{\Chamber}$ be the retraction onto $\Apartment$ centered at $\Chamber$.

The main theorem of this section (Theorem~\ref{thm:convex_sets}) involves a technical condition on subsets $\ClosedSet$ of $\Apartment$ relative to $\Chamber$. To motivate it, assume first that $\Space$ is euclidean. If $\ClosedSet$ is closed and convex, then for every point $a \in \partial A$ there is a supporting hyperplane. In particular, there is a vector $n_a \in \Link{\Space}a$ such that every direction at $a$ that points into $A$ includes a nonobtuse angle with $n_a$. We say that $n_a$ is an \emph{anti-normal vector}. If in addition $\ClosedSet$ has nonempty interior, then $n_a$ can be chosen to point into $A$, that is, to lie in $\Link{A}a$. The additional condition that we impose is that it can also be chosen to point into the projection of $C$, see Figure~\ref{fig:normal_condition}:

\begin{figure}[htb]
\begin{center}
\includegraphics[scale=.5]{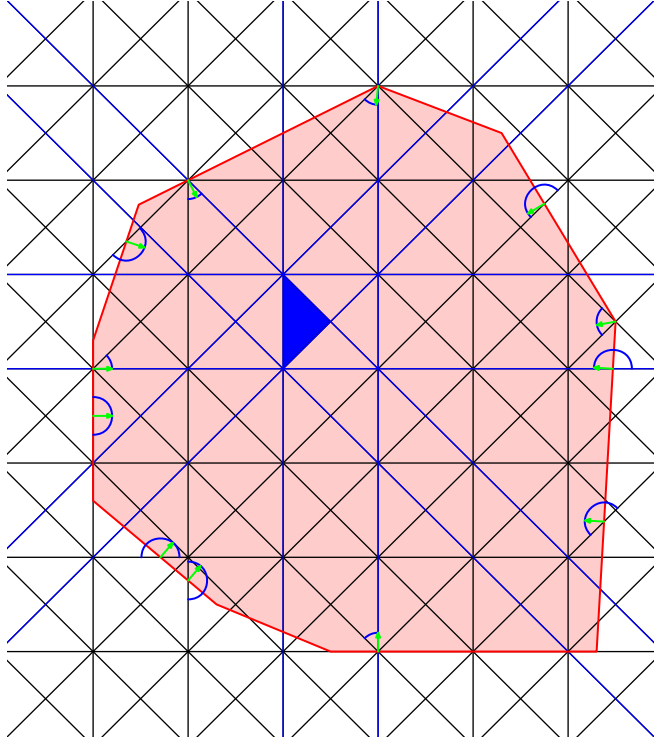}
\end{center}
\caption{A Coxeter complex of type $\tilde{B}_2$ with a distinguished chamber (blue). The set drawn in red satisfies the normal condition. At some boundary points an anti-normal vector (green) is drawn that points into the projection of $\Chamber$ to the link of that boundary point (blue).}
\label{fig:normal_condition}
\end{figure}

Formally, we say that a closed $D_\kappa$-convex subset $\ClosedSet$ of $\Apartment$ \emph{satisfies the normal condition (with respect to $\Chamber$)} if for every $a \in \partial A$ there is an $n_a \in \Link{\ClosedSet}a \intersect \Proj{\Link{}a}\Chamber$ such that $\Link{A}a \subseteq \CBall{}(n_\Point,\pi/2)$.

We start with some elementary observations which will facilitate the proof of Theorem~\ref{thm:convex_sets}.

\begin{obs}
\label{obs:apartment_not_distinguished}
Consider a set $\ClosedSet \subseteq \Apartment$ and set $\tilde{\ClosedSet} \defeq \rho^{-1}(\ClosedSet)$. If $\Apartment'$ is another apartment that contains $\Chamber$, then $\tilde{\ClosedSet} = \Retraction{\Apartment'}{\Chamber}^{-1}(\ClosedSet')$ where $\ClosedSet' \defeq \tilde{\ClosedSet} \intersect \Apartment'$ is the isometric image of $\ClosedSet$ under $\Retraction{\Apartment'}{\Chamber}$.
\end{obs}

If $a \in \Space$ is a point, then $\rho$ induces a map $\Link{\Space}a \to \Link{\Apartment}\rho(a)$ given by taking the direction from $a$ to $x$ to the direction from $\rho(a)$ to $\rho(x)$. Let us denote this map by $\rho|_{\Link{\Space}a}$.

\begin{obs}
\label{obs:retraction_induces_retraction}
Let $a \in \Apartment$ and set $\overbar{\Apartment} \defeq \Link{\Apartment} a$ and $\overbar{\Chamber} \defeq \Proj{\Link{}a}\Chamber$. Then $\rho|_{\Link{}a} = \Retraction{\overbar{\Apartment}}{\overbar{\Chamber}}$.
\end{obs}

\begin{obs}
\label{obs:retraction_injective_on_projection}
Let $a \in \Space$ be arbitrary and let $a' \defeq \rho(a)$. For every $\alpha' \in \Proj{\Link{}a'} \Chamber$ there is a unique $\alpha \in \Proj{\Link{}a} \Chamber$ with $\rho|_{\Link{}a}(\alpha) = \alpha'$.
\end{obs}

\begin{lem}
\label{lem:normal_condition_hereditary}
Let $\ClosedSet \subseteq \Apartment$ be closed and polyhedral. Let $a \in \partial A$ and set $\overbar{\Apartment} \defeq \Link{\Apartment} a$ and $\overbar{\Chamber} \defeq \Proj{\Link{}a}\Chamber$. If $\ClosedSet$ satisfies the normal condition, then $\overbar{\ClosedSet}$ satisfies the normal condition (with respect to $\overbar{\Chamber}$).
\end{lem}

\begin{proof}
Let $\varepsilon > 0$ be such that the cone condition \eqref{eq:cone_condition} is satisfied. Let $\beta \in \partial \overbar{\ClosedSet}$. Let $b \in \partial \ClosedSet$ be a point in direction $\beta$ from $a$ at distance at most $\varepsilon/2$. Let $c$ be a point in direction $n_b$ from $b$ at distance at most $\varepsilon/2$. Let $\gamma \in \Link{\Space}a$ be the direction from $a$ to $c$. Then the direction from $\beta$ to $\gamma$ is a possible $n_\beta$.
\end{proof}

\begin{thm}
\label{thm:convex_sets}
Let $\Space$ be a locally compact metric building and let $\kappa \defeq 1$ if $\Space$ is spherical and $\kappa \defeq 0$ if $\Space$ is euclidean. Let $\Apartment \subseteq \Space$ be an apartment and let $\Chamber \subseteq \Apartment$ be a chamber. Let $\ClosedSet \subseteq \Apartment$ be closed, polyhedral, $D_\kappa$-convex and assume $\ClosedSet \intersect \Chamber \ne \emptyset$ and that $\ClosedSet$ satisfies the normal condition.
Then $\tilde{\ClosedSet} \defeq \Retraction{\Apartment}{\Chamber}^{-1}(\ClosedSet)$ is locally convex.
\end{thm}

The following example illustrates why the normal condition is necessary.

\begin{exmpl}
Let $\Space$ be a thick building of type $\tilde{A}_2$ and pick an apartment $\Apartment$ and a chamber $\Chamber \subseteq \Apartment$ arbitrarily. Let $\ClosedSet$ be a geodesic segment that passes through one vertex of $\Chamber$, call it $\Point$, and meets the opposite edge perpendicularly. Then $\tilde{\ClosedSet} = \Retraction{\Apartment}{\Chamber}^{-1}(\ClosedSet)$ is not locally convex: there are two chambers, say $D$ and $D'$, that are opposite $\Chamber$ in $\Link\Point$ and adjacent to each other. Let $\AltPoint$ and $\AltPoint'$ be interior points of $D$ respectively $D'$ that lie in $\tilde\ClosedSet$. Then the directions from $\Point$ to $\AltPoint$ and from $\Point$ to $\AltPoint'$ include an angle of $\pi/3$ (as can be easily seen in an apartment that contains $D$ and $D'$). Hence the geodesic from $\AltPoint$ to $\AltPoint'$ leaves $\ClosedSet$.
\end{exmpl}

\begin{proof}
The proof is by induction on the dimension. For $\dim \Space = 0$ there is nothing to show.

For any apartment $\Apartment'$ that contains $\Chamber$ we may consider the set $\ClosedSet' \defeq \tilde{\ClosedSet} \intersect \Apartment'$ and have $\tilde{\ClosedSet} = \Retraction{\Apartment'}{\Chamber}^{-1}(\ClosedSet')$ by Observation~\ref{obs:apartment_not_distinguished}. What is more, if $a' \in \partial A'$, then $\Retraction{\Apartment}{\Chamber}(a') \eqdef a$ has a direction $n_{a}$ as in the normal condition and by Observation~\ref{obs:retraction_injective_on_projection} this gives rise to a direction $n_{a'} \in \Link{\Apartment'}a'$ showing that $A'$ satisfies the normal condition. This shows that instead of $\Apartment$ we may consider any apartment that contains $\Chamber$ and the hypotheses of the theorem are satisfied.

To show that $\tilde{\ClosedSet}$ is locally convex, it suffices by  Proposition~\ref{prop:local_convexity_of_polyhedral_sets}~\eqref{item:local_convexity_of_polyhedral_sets} to show that $\Link{\tilde{\ClosedSet}}a$ is $\pi$-convex for every $a \in \partial \tilde{\ClosedSet}$.

So let $a \in \partial\tilde{\ClosedSet}$. By our above discussion we may assume that $a \in \Apartment$. Let $\overbar{\ClosedSet} \defeq \Link{\ClosedSet}a$, $\overbar{\Apartment} \defeq \Link{\Apartment}a$, and $\overbar{\Chamber} \defeq \Proj{\Link{}a}\Chamber$. By Observation~\ref{obs:retraction_induces_retraction} we have $\Link{\tilde{\ClosedSet}}a = \Retraction{\overbar{\Apartment}}{\overbar{\Chamber}}^{-1}(\overbar{\ClosedSet})$. We want to apply the induction hypothesis to show that $\Retraction{\overbar{\Apartment}}{\overbar{\Chamber}}^{-1}(\overbar{\ClosedSet})$ is locally convex. To do so, we have to verify that the hypotheses are met.

First, $\overbar{\ClosedSet}$ is closed and polyhedral by Proposition~\ref{prop:local_convexity_of_polyhedral_sets}~\eqref{item:polyhedral_hereditary}. It is $\pi$-convex by Proposition~\ref{prop:local_convexity_of_polyhedral_sets}~\eqref{item:local_convexity_of_polyhedral_sets}. Next, $n_a \in \overbar{\Chamber} \intersect \overbar{\ClosedSet}$, so that this intersection is nonempty. Finally, $\overbar{\ClosedSet}$ satisfies the normal condition by Lemma~\ref{lem:normal_condition_hereditary}. So we can indeed apply the theorem inductively and get that $\Link{\tilde{\ClosedSet}}a$ is locally convex.

To see that it is $\pi$-convex, using Theorem~\ref{thm:local_convexity}, it remains to show that the diameter of $\Link{\tilde{\ClosedSet}}a$ is at most $\pi$ in the length metric. We claim that in fact every direction $\gamma$ in $\Link{\tilde{\ClosedSet}}a$ has distance at most $\pi/2$ from $n_a$. This is clear by choosing an apartment that contains $\gamma$ and $\Chamber$ (which will automatically contain $n_a \in \Proj{\Link{}a}\Chamber$).
\end{proof}

We get the following special case for euclidean buildings.

\begin{thm}
\label{thm:convex_sets_euclidean}
Let $\Space$ be a locally compact euclidean building, let $\Apartment \subseteq \Space$ be an apartment and let $\Chamber \subseteq \Apartment$ be a chamber. Let $\ClosedSet \subseteq \Apartment$ be closed, polyhedral, convex and assume that $\ClosedSet \intersect \Chamber \ne \emptyset$ and that $\ClosedSet$ satisfies the normal condition. Then $\tilde{\ClosedSet} \defeq \Retraction{\Apartment}{\Chamber}^{-1}(\ClosedSet)$ is convex.
\end{thm}

\begin{proof}
We can apply Theorem~\ref{thm:convex_sets} to obtain that $\tilde{\ClosedSet}$ is locally convex. It is also connected because $\ClosedSet$ is connected and $\ClosedSet \intersect \Chamber \ne \emptyset$: Let $\Point \in \ClosedSet \intersect \Chamber$. If $\AltPoint \in \tilde{\ClosedSet}$ is arbitrary let $\Apartment'$ be an apartment that contains $\Chamber$ and $\AltPoint$. By Observation~\ref{obs:apartment_not_distinguished} there is a path from $\AltPoint$ to $\Point$ in $\tilde{\ClosedSet}$. So we can deduce from Theorem~\ref{thm:local_convexity_cat0} that $\tilde{\ClosedSet}$ is convex.
\end{proof}

We now collect the basic facts that are needed to formulate Theorem~\ref{thm:convex_sets_euclidean} with $\Chamber$ replaced by a chamber at infinity.
If $\Space$ is a euclidean building, then the visual boundary $\Infty\Space$ is a spherical building, called the \emph{building at infinity}. Its apartments are the visual boundaries of apartments of $\Space$. We have the following parallel to Fact~\ref{fact:buildings}.

\begin{fact}
Let $\Space$ be a euclidean building.
\label{fact:euclidean_buildings}
\begin{enumerate}
\item Given an apartment $\Apartment$ of $\Space$ and a chamber at infinity $\ChamberAtInfinity \subseteq \Infty\Apartment$ there is a map $\Retraction{\Apartment}{\ChamberAtInfinity} \colon \Space \to \Apartment$ called the \emph{retraction onto $\Apartment$ centered at $\ChamberAtInfinity$} such that $\Retraction{\Apartment}{\ChamberAtInfinity}|_\Apartment = \id_\Apartment$ and $\Retraction{\Apartment}{\ChamberAtInfinity}|_{\Apartment'}$ is an isometry for every apartment $\Apartment'$ that contains $\ChamberAtInfinity$ in its boundary. Moreover, $\Retraction{\Apartment}{\ChamberAtInfinity} \circ \Retraction{\Apartment'}{\ChamberAtInfinity} = \Retraction{\Apartment}{\ChamberAtInfinity}$.\label{item:retractions_from_infinity}
\item Let $\ChamberAtInfinity \subseteq \Infty\Space$ be a chamber at infinity and let $\Cell \subseteq \Space$ be an abritrary cell. There is a unique chamber $\Proj{\Cell}\ChamberAtInfinity \ge \Cell$, called the \emph{projection of $\ChamberAtInfinity$ onto $\Cell$} such that every apartment that contains $\ChamberAtInfinity$ and $\Cell$ also contains $\Proj{\Cell}\ChamberAtInfinity$. If $\Apartment$ is an apartment that contains $\ChamberAtInfinity$ in its boundary, then $\Proj{\Retraction{\Apartment}{\ChamberAtInfinity}\Cell}\ChamberAtInfinity = \Retraction{\Apartment}{\ChamberAtInfinity}(\Proj{\Cell}\ChamberAtInfinity)$.\label{item:projections_from_infinity}
\end{enumerate}
\end{fact}

In the discussion before Theorem~\ref{thm:convex_sets} we may replace the chamber $\Chamber$ by a chamber at infinity $\ChamberAtInfinity$. The normal condition translates literally: A closed convex subset $\ClosedSet$ of $\Apartment$ \emph{satisfies the normal condition (with respect to $\ChamberAtInfinity$)} if for every $a \in \partial A$ there is an $n_a \in \Link{\ClosedSet}a \intersect \Proj{\Link{}a}\ChamberAtInfinity$ such that $\Link{A}a \subseteq \CBall{}(n_\Point,\pi/2)$. Observation~\ref{obs:apartment_not_distinguished} to Lemma~\ref{lem:normal_condition_hereditary} remain true with $\Chamber$ replaced by $\ChamberAtInfinity$, so we get:

\begin{thm}
\label{thm:convex_sets_euclidean_from_infinity}
Let $\Space$ be a locally compact euclidean building, let $\Apartment \subseteq \Space$ be an apartment and let $\ChamberAtInfinity \subseteq \Infty\Apartment$ be a chamber at infinity. Let $\ClosedSet \subseteq \Apartment$ be closed, polyhedral, convex and assume that $\Infty\ClosedSet \intersect \ChamberAtInfinity \ne \emptyset$ and that $\ClosedSet$ satisfies the normal condition. Then $\tilde{\ClosedSet} \defeq \Retraction{\Apartment}{\ChamberAtInfinity}^{-1}(\ClosedSet)$ is convex.
\end{thm}

\begin{proof}
Theorem~\ref{thm:convex_sets} holds analogously to show that $\tilde{\ClosedSet}$ is locally convex. For connectedness let $\Point \in \ClosedSet \intersect \Apartment$ and $\Point' \in \ClosedSet$. Let $\Apartment'$ be an apartment that contains $a'$ and is such that $\ChamberAtInfinity \subseteq \Apartment'\Infty{}$. Let $\gamma$ be a geodesic ray in $\ClosedSet$ with limit point in $\ChamberAtInfinity$.

That $\Apartment$ and $\Apartment'$ contain $\ChamberAtInfinity$ in their boundary means that they contain sectors $S$ respectively $S'$ with $\Infty{S} = \ChamberAtInfinity = S'\Infty{}$. For large enough $t$ we have $\AltPoint \defeq \gamma(t) \in S \intersect S'$. Then $[\Point, \AltPoint] \union [\AltPoint,\Point']$ lies in $\ClosedSet$ and connects $\Point$ to $\AltPoint$.
\end{proof}

\section{Alternative proof of the application}
\label{sec:alternative_proof}

In this section we present an alternative proof of Theorems~\ref{thm:convex_sets_euclidean} (which works analogously for Theorem~\ref{thm:convex_sets_euclidean_from_infinity}). It needs a weaker version of the normal condition and dispenses with the assumption of $\ClosedSet$ being polyhedral.

Before we formulate the weak normal condition, let us reformulate the normal condition. Let $\Apartment$ be a euclidean Coxeter complex and let $\Chamber \subseteq \Apartment$ be a chamber. Recall that if $\Cell \subseteq \Apartment$ is a simplex, then the projection $\Proj{\Cell}\Chamber$ can be characterized as the chamber that is separated from $\Chamber$ by every wall that strictly separates $\Cell$ and $\Chamber$. Therefore a vector $\gamma \in \Link\Point$ points into $\Proj{\Link\Point}\Chamber$ if for every wall containing $\Point$, it points into the halfspace that contains $\Chamber$. The weak normal condition differs from the normal condition by not requiring any more that the anti-normal vector point into $\ClosedSet$. For technical reasons we also formulate the statement for normal vectors rather than anti-normal vectors. The definition is that a closed, convex subset $\ClosedSet \subseteq \Apartment$ satisfies the \emph{weak normal condition} if for every $a \in \partial \ClosedSet$ and every 
wall containing $a$ there is a normal vector to $\ClosedSet$ at $a$ pointing into the halfspace of $H$ that does not contain $\Chamber$.

The reason for formulating the weak normal condition in terms of hyperplanes is the following observation.

\begin{obs}
\label{obs:key_observation}
Let $H^+$ be a (closed) halfspace in euclidean space and let $H = \partial H^+$. Let $\ClosedSet$ be a closed convex set not fully contained in $H^+$. If for every $a \in \partial \ClosedSet \intersect H$ there is a normal vector pointing into $H^+$, then
\[
H^+ \intersect \ClosedSet \subseteq (H \intersect \ClosedSet) + H^\perp \text{ .}
\]
\end{obs}

\begin{rem}
Observation~\ref{obs:key_observation} implies that it suffices to check the normal condition along the two walls that are closest to $\Chamber$ in their parallelity class. In Figure~\ref{fig:normal_condition} these are drawn in blue.
\end{rem}

In what follows we will denote by $\ClosedSet + \varepsilon$ the closed set of all points at distance at most $\varepsilon$ from $\ClosedSet$.

\begin{cor}
\label{cor:normal_condition_preserved}
Let $\Apartment$ be a euclidean Coxeter complex, let $\Chamber \subseteq \Apartment$ be a chamber, and let $\ClosedSet$ be a closed, convex set such that $\ClosedSet \intersect \Chamber \ne \emptyset$. If $\ClosedSet$ satisfies the weak normal condition, then $\ClosedSet + \varepsilon$ satisfies the weak normal condition for every $\varepsilon > 0$.
\end{cor}

\begin{proof}
Let $x \in \partial(\ClosedSet + \varepsilon)$ and let $a$ be its projection to $\ClosedSet$. Then $v = x-a$ is the (unique up to scaling) normal vector of $\ClosedSet + \varepsilon$ at $x$.

Assume that $x$ lies in a wall $H$ and let $H^+$ denote the closed halfspace that does not contain $\Chamber$. To show that $v$ points into $H^+$ means to show that $a$ does not lie in the interior of $H^+$. But this follows from the Observation because otherwise the projection of $a$ to $H$ would lie in $A$ and be closer to $x$, a contradiction.
\end{proof}

Returning to the setting of the theorem, let $\Space$ be a euclidean building, $\Apartment$ be an apartment, and let $\rho$ be the retraction onto $\Apartment$ centered at a chamber $\Chamber$.
If $\gamma$ is a geodesic, then $\rho \circ \gamma$ is a piecewise linear path. More precisely it is locally a geodesic on chambers and may or may not be locally a geodesic where it hits a wall.

Let $\ClosedSet \subseteq \Apartment$ be a closed, convex set. We call a geodesic $\gamma \colon [0,l] \to \Space$ \emph{ascending} if $d_\ClosedSet \circ \rho \circ \gamma$ is strictly ascending where $d_\ClosedSet(x) = d(x,\ClosedSet)$.

Note that whether or not $\gamma$ is ascending can be determined using linear
algebra: turn $\Apartment$ into a euclidean vector space by picking $\rho \circ
\gamma(t)$ as origin. Let $m$ be the vector pointing away from $\ClosedSet$.
Note that this is the normal vector at $\rho \circ \gamma(t)$ of $\ClosedSet +
d_\ClosedSet(\rho\circ\gamma(t))$. Let $v$ and $w$ be the incoming
and outgoing tangent vectors to $\rho \circ \gamma$ at $t$. Then $\gamma$ is
ascending on a neighborhood of $t$ if and only if both $v$ and $w$ include a
nonobtuse angle with $m$, i.e.\ $(v,m),(w,m) \ge 0$.

\begin{prop}
\label{prop:ascending_goes_on}
Let $\Space$, $\Apartment$, $\rho$, $\ClosedSet$ be as above and assume that $\ClosedSet \intersect \Chamber \ne \emptyset$. If $\gamma \colon [0,l] \to \Space$ is ascending, then so is every geodesic (ray) $\bar{\gamma}$ extending it.
\end{prop}

\begin{proof}
Assume first that $\rho \circ \bar{\gamma}$ meets no cells of codimension $2$ and is not contained in a wall.

Let $t \ge l$ be minimal such that $\bar{\gamma}(t)$ lies in a wall $H$. Let
$\delta > 0$ be sufficiently small that $\bar{\gamma}(t-\delta,t+\delta)$ meets
no other wall. If $\Apartment'$ is an apartment that contains $\Chamber$ and
$\bar{\gamma}(t-\delta,t]$, then $\rho = \rho \circ
\Retraction{\Apartment'}{\Chamber}$. Therefore we may as well assume that
$\bar{\gamma}(t-\delta,t] \subseteq \Apartment$. It suffices to show that
$\bar{\gamma}(t-\delta,t+\delta)$ is ascending.

To this end let $n$ be the normal vector in $\Apartment$ of $H$ pointing away from $\Chamber$. Let $m$ be the vector in $\Apartment$ pointing away from $\ClosedSet$. By the weak normal condition and Corollary~\ref{cor:normal_condition_preserved} $(m,n) \ge 0$. Let $v$ be the incoming and $w$ the outgoing tangent vector to $\rho \circ \bar{\gamma}$ at $t$. We know that $(m,v) \ge 0$ because $\gamma$ is ascending.

Now either $w=v$ and there is nothing to show. Or $w = v + \lambda n$ for some
$\lambda > 0$ and $(m,w) = (m,v) + (m,n) \ge 0$. Thus also in that case
$\bar{\gamma}$ is ascending on a neighborhood of $t$ finishing the proof under
the initial assumptions.

If $\bar{\gamma}$ is a general geodesic, then there is a sequence $\gamma_i$ of
geodesics meeting these assumptions and converging to $\bar{\gamma}$. Moreover,
the $\gamma_i$ can be chosen so that an initial segment of each $\gamma_i$ is
contained in a chamber that also contains an initial segment of $\bar{\gamma}$.
They are therefore initially ascending and therefore ascending by the previous
discussion. It follows that the limit $\bar{\gamma}$ is ascending.
\end{proof}

\begin{thm}
\label{thm:strong_convex_sets_euclidean}
Let $\Space$ be a euclidean building, let $\Apartment \subseteq \Space$ be an apartment and let $\Chamber \subseteq \Apartment$ be a chamber. Let $\ClosedSet \subseteq \Apartment$ be closed and convex and assume that $\ClosedSet \intersect \Chamber \ne \emptyset$ and that $\ClosedSet$ satisfies the weak normal condition. Then $\tilde{\ClosedSet} \defeq \Retraction{\Apartment}{\Chamber}^{-1}(\ClosedSet)$ is convex.
\end{thm}

\begin{proof}
Let $\rho \defeq \Retraction{\Apartment}{\Chamber}$. Let $\gamma \colon [0,l] \to \Space$ be a geodesic with $\rho(\gamma(0)) \in \ClosedSet$ and $\rho(\gamma(l)) \in \ClosedSet$. If there was a $t \in [0,l]$ with $\rho(\gamma(t)) \nin \ClosedSet$, then $\gamma$ would be ascending on some subinterval. But in view of Proposition~\ref{prop:ascending_goes_on} this would contradict $\rho(\gamma(l)) \in \ClosedSet$.
\end{proof}

\providecommand{\bysame}{\leavevmode\hbox to3em{\hrulefill}\thinspace}
\providecommand{\MR}{\relax\ifhmode\unskip\space\fi MR }
\providecommand{\MRhref}[2]{%
  \href{http://www.ams.org/mathscinet-getitem?mr=#1}{#2}
}
\providecommand{\href}[2]{#2}

\end{document}